\documentclass[reqno, 12pt]{amsart}

\usepackage{mathrsfs}
\usepackage{amscd}
\usepackage{amsmath}
\usepackage{latexsym}
\usepackage{amsfonts}
\usepackage{amssymb}
\usepackage{amsthm}
\usepackage{graphicx}
\usepackage{hyperref}
\usepackage{makecell}
\usepackage{array}
\usepackage{booktabs}
\usepackage{multirow}
\usepackage{color,xcolor}

\newtheorem{theorem}{Theorem}[section]

\newtheorem{corollary}[theorem]{Corollary}
\newtheorem{lemma}[theorem]{Lemma}

\numberwithin{equation}{section}

\title[Inverse source problem]{Inverse source scattering problem for a nonlinear Schr\"odinger equation}

 \author[L. Zhang]{Lei Zhang}
 \address{Department of Mathematics, Zhejiang University of Technology, Hangzhou 310014, China.}
\email{zhanglei@zjut.edu.cn}
  
 \author[Y. Zhao]{Yue Zhao}
\address{School of Mathematics and Statistics, Central China Normal University,
Wuhan 430079, China.}
\email{zhaoyueccnu@163.com}


\begin{document}

\begin{abstract}

We study an inverse source scattering problem for the Schr\"odinger equation with a quadratic nonlinearity. 
In general, uniqueness of inverse source problems can not be guaranteed at a fixed energy. 
Therefore, additional information is required for the source in order to
obtain a unique solution.
By adding reference point sources, we show that 
a general source function could be uniquely determined from boundary measurements at a fixed wavenumber. This method does not apply to inverse source problems of linear equations since it uses the nonlinearity as a tool. The proof utilizes the method of linearization to reduce the nonlinear inverse source scattering problem to the inverse potential scattering problem of the linear Schr\"odinger equation. 

\end{abstract}

\maketitle

\section{Introduction}

This paper is concerned with an inverse source scattering problem for the Schr\"odinger equation with a quadratic nonlinearity.  
For linear equations, in general the uniqueness of the inverse source problems can not be guaranteed at a fixed frequency due to the existence of non-radiating
sources \cite{Monk, Bao, BC}. We study the question of whether the unknown source can be uniquely determined by boundary measurements at a fixed frequency for the nonlinear Schr\"odinger equation.

We consider the following nonlinear Schr\"odinger equation with a quadratic nonlinearity in two dimensions
\begin{equation}\label{eqn}
-\Delta u - \kappa^2 u + \alpha u^2 = f,
\end{equation}
where $f(x)\in L_{\rm comp}^2(\mathbb R^2)$ is the external source, $\alpha(x)\in L_{\rm comp}^\infty(\mathbb R^2)$ is the potential, $u(x)$ is the radiating field and $\kappa>0$ is the wavenumber.  Let $B_R := \{x\in\mathbb R^2 ~:~ |x|< R\}$ be the disk of radius $R>0$ with boundary $\partial B_R$. 
We assume that both $f$ and $\alpha$ are real-valued and have compact supports contained in $B_R$. We also assume that $\text{supp}f\subset\text{supp}\alpha$.

Let $\epsilon_{0}>0$ be a positive constant.
We are interested in the following inverse problem:

\noindent \textbf{IP}. Determine the source function $f$ from the following boundary measurements at a fixed frequency
\begin{equation}\label{data}
\{\tilde{u}_\epsilon(x, x_0)\vert_{\partial B_R}: \text{for all} \, \epsilon\in [0, \epsilon_{0}), \, x_0\in\partial B_R\}.
\end{equation}

\noindent
Here $\tilde{u}_\epsilon(x, x_0)$ denotes the radiating solution corresponding to the perturbed source function 
\[
f_{\epsilon, x_0} = f + \epsilon\delta(x - x_0),
\]
which satisfies
\[
-\Delta \tilde{u}_\epsilon - \kappa^2 \tilde{u}_\epsilon + \alpha \tilde{u}_\epsilon^2 = f_{\epsilon, x_0}.
\]
In fact, if one only takes boundary measurements radiated by the source function $f$, the uniqueness can not be guaranteed. 
For example, choose $\varphi\in C_0^\infty(B_R)$ such that $f_0:= -\Delta \varphi - \kappa^2 \varphi + \alpha\, \varphi^2 \neq 0$. 
It is easy to see that the $\varphi$ vanishes near $\partial B_R$, which implies that $f_0$ can not be recovered in this case. 
Thus, additional information are needed to guarantee the uniqueness. 
In this paper we collect the additional boundary measurements by placing reference point sources on $\partial B_R$ with amplitudes varying in a small interval.  
In this way we are able to utilize the method of linearization to reduce the inverse source problem
to an inverse potential problem of the linear Schr\"odinger equation. The uniqueness of the inverse source problem follows from the existing results for the inverse potential problems of linear Schr\"odinger equations. In this paper, we focus on the two-dimensional case because the direct problem is studied in two dimensions. We use the decay of the fundamental solution with respect to the wavenumber in two dimensions to construct a sequence which converges to the unique solution. This decay property does not hold in three dimensions. However, the method for the study of the inverse problem may be extended to higher dimensions once we have the well-posedness of the direct problem.

We briefly review the existing literature on inverse source scattering problems. For linear Helmholtz equations, 
it has been realized that the uniqueness for the inverse source problem can be regained by using multi-frequency boundary measurements \cite{BLLT, Bao}.
There are also various studies on the increasing stability on this subject. See e.g. \cite{blz, CIL, LZZ, IL, LSX} and references therein.
For the inverse source problem of certain semilinear elliptic equations by Dirichlet-to-Neumann map in a bounded domain, we mention the reference \cite{Myers}.

The paper is organized as follows. The direct scattering problem is analyzed in Section \ref{dp}. Section \ref{ip} is devoted to the inverse source problem.

\section{Direct scattering problem}\label{dp}

To study the inverse problem we shall first study the direct problem. In this section we consider a more general equation and study its well-posedness. 
The proof adapts the arguments in \cite{Servo}. Compared with \cite{Servo}, we further derive an explicit resolvent estimate with respect to the wavenumber $\kappa$, which is of independent interest. Hereafter, the notation $C$ stands for a generic positive constant depending on $\|f\|_{L^2(B_R)}$ which may change step
by step.

Consider the following nonlinear Schr\"odinger equation
\begin{equation}\label{main_eq}
-\Delta u - \kappa^2 u + Vu + \alpha u^2 = f.
\end{equation}
Here $V(x)$ and $\alpha(x)$ are bounded real-valued potential functions compactly supported in $B_R$.
The outgoing solution to \eqref{main_eq} should satisfy the following Lippmann--Schwinger integral equation
\begin{equation}\label{ie}
u =  \int_{\mathbb R^2} H^{(1)}_0(\kappa|x-y|) f {\rm d}y - \int_{\mathbb R^2} H^{(1)}_0(\kappa|x-y|) \Big( V u + \alpha u^2 \Big) {\rm d}y.
\end{equation}
Introduce the sequence
\begin{align}\label{sequence}
u_{j + 1} = u_0 - \int_{\mathbb R^2} H^{(1)}_0(\kappa|x-y|) \Big(V u_{j} + \alpha u^2_{j} \Big) {\rm d}y, \quad j\geq 0.
\end{align}
Here $u_0 = \int_{\mathbb R^2} H^{(1)}_0(\kappa|x-y|) f {\rm d}y\in H^2_{loc}(\mathbb R^2)$. We will prove the well-posedness of the direct scattering problem
by showing that the sequence $u_j$ converges to a unique bounded solution of \eqref{main_eq}.

As $H^{(1)}_0(\kappa|x-y|)$ has the following expression \cite{FY06}
\begin{align*}
H^{(1)}_0(\kappa|x-y|) =  C e^{{\rm i} \kappa |x - y|} \int_0^\infty e^{-t} t^{-\frac{1}{2}} \Big( \frac{t}{2} - {\rm i}\kappa |x - y| \Big)^{-\frac{1}{2}} {\rm d}t,
\end{align*}
where $C$ is a positive constant, a simple calculation yields 
\begin{align}\label{kernel_1}
|H^{(1)}_0(\kappa|x-y|)| \leq  \frac{C}{|\kappa|^{\frac{1}{2}}|x - y|^{\frac{1}{2}}} \int_0^\infty e^{-t} t^{-\frac{1}{2}} {\rm d}t \leq  
\frac{C}{|\kappa|^{\frac{1}{2}}|x - y|^{\frac{1}{2}}}.
\end{align}
The estimate \eqref{kernel_1} for $H^{(1)}_0(\kappa|x-y|)$ will be useful in the subsequent analysis.

\begin{lemma}\label{uj}
There exists $C_0>0$ such that
for $|\kappa|\geq C_0$ and $j\geq 0$ one has 
\begin{align}\label{est}
\|u_j\|_{L^\infty(\mathbb R^2)} \leq C \Big(1 + \frac{1}{|\kappa|^{1/2}}\Big).
\end{align}
\end{lemma}

\begin{proof}

The proof is carried out by induction. Clearly \eqref{est} holds for $j = 0$. Assume that \eqref{est} holds for $u_j$.
Since 
\begin{align*}
u_{j + 1}~ = \int_{\mathbb R^2} H^{(1)}_0(\kappa|x-y|) f {\rm d}y - \int_{\mathbb R^2} H_0^{(1)}(\kappa|x - y|) \Big(V u_{j} + \alpha u^2_{j} \Big) {\rm d}y
\end{align*}
and $V, \alpha\in L^{\infty}_{\rm comp}(\mathbb R^2)$, one has that
\[
\|u_{j + 1}\|_{L^\infty(\mathbb R^2)} \leq  C + C \int_{B_R} |H^{(1)}_0(\kappa|x-y|)| (|u_j| + |u_j|^2){\rm d}y.
\]
Moreover, by the estimates \eqref{kernel_1} and \eqref{est} for $u_j$ we have that
\[
\|u_{j + 1}\|_{L^\infty(\mathbb R^2)} %\leq C + \frac{C}{|\kappa|^{1/2}} \Big(1 + \frac{1}{|\kappa|^{1/2}}\Big)^3 
\leq C\Big(1 + \frac{1}{|\kappa|^{1/2}}\Big)
\]
where the last inequality holds  when $|\kappa|$ is large. 
This completes the proof.

\end{proof}

\begin{lemma}\label{cauchy}
It holds that 
\[
\|u_{j + 1} - u_j\|_{L^\infty(\mathbb R^2)} \leq \eta (\|u_{j} - u_{j - 1}\|_{L^\infty(\mathbb R^2)}),
\]
where $\eta = \frac{C}{|\kappa|^{1/2}} \Big(1 + \frac{1}{|\kappa|^{1/2}}\Big)$.
\end{lemma}

\begin{proof}

By a direction calculation one has
\begin{align}\label{est1}
|u_{j + 1} - u_j| \leq C \int_{B_R} |H^{(1)}_0(\kappa|x-y|) | \Big( |u_j - u_{j-1}| + |u^2_{j} - u^2_{j-1}| \Big){\rm d}y.
\end{align}
Moreover, one has from Lemma \ref{uj} that
\begin{align*}
 |u^2_j- u^2_{j - 1} | &\leq |u_j - u_{j-1}| |u_j + u_{j-1}|\\
&\leq C \Big(1 + \frac{1}{|\kappa|^{1/2}}\Big) |u_j - u_{j - 1}|,
\end{align*}
which by \eqref{est1} and \eqref{kernel_1} gives 
\begin{align*}
\|u_{j + 1} - u_j\|_{L^\infty(\mathbb R^2)} &\leq \frac{C}{|\kappa|^{1/2}} \Big(1 + \frac{1}{|\kappa|^{1/2}}\Big) 
\|u_{j}  - u_{j - 1} \|_{L^\infty(\mathbb R^2)} \\
\end{align*}
The proof is completed.
\end{proof}

\begin{theorem}\label{bs}
There exists $C_0>0$ such that the sequence $u_j$ converges to $u$ in $L^\infty(\mathbb R^2)$ for $|\kappa|\geq C_0$, which satisfies 
the Lippmann--Schwinger integral equation \eqref{ie}. Moreover, the solution $u$ has the form $u = \int_{\mathbb R^2} H^{(1)}_0(\kappa|x-y|) f {\rm d}y + h$ where
\[
\|h\|_{L^\infty(\mathbb R^2)} = \mathcal{O} \Big(\frac{1}{|\kappa|^{1/2}}\Big).
\]
\end{theorem}

\begin{proof}

We first show that $u_j$ is a Cauchy sequence in $L^\infty(\mathbb R^2)$. We have from Lemma \ref{cauchy}
\begin{align*}
\|u_m - u_n\|_{L^\infty(\mathbb R^2)} &\leq \|u_m - u_{m - 1}\|_{L^\infty(\mathbb R^2)} + \cdots 
+ \|u_{n + 1} - u_{n}\|_{L^\infty(\mathbb R^2)}\\
&\leq (\eta^{m - 1} + \cdots + \eta^{n}) \|u_1 - u_{0}\|_{L^\infty(\mathbb R^2)} \\
&\leq \frac{\eta^n}{1 - \eta} \|u_1 - u_{0}\|_{L^\infty(\mathbb R^2)},
\end{align*}
where $m>n$ and $\eta = \frac{C}{|\kappa|^{1/2}} \Big(1 + \frac{1}{|\kappa|^{1/2}}\Big)$. Choosing $|\kappa|$ large such that $\eta<1$ one has 
that $u_j$ is indeed a Cauchy sequence in $L^\infty(\mathbb R^2)$. Denote its limit by $u$. Then $u$ is a bounded solution to \eqref{ie} 
from the dominated convergence theorem and \eqref{sequence}. One also has from \eqref{sequence} that
\[
h = -  \int_{\mathbb R^2} H_0^{(1)}(\kappa|x - y|) \Big(V u + \alpha u^2 \Big) {\rm d}y.
\]
Finally, the proof is completed by noticing
\[
|h| \leq C\int_{B_R}  |H_0^{(1)}(\kappa|x - y|)| (|u| + |u|^2 ) {\rm d}y \leq \frac{C}{|\kappa|^{1/2}}.
\]
\end{proof}
As a corollary of Theorem \ref{bs} and elliptic regularity theory, we have the well-posedness of the direct scattering problem.
\begin{corollary}\label{well}
There exists $C_0>0$ such that for $\kappa\geq C_0$ the nonlinear Schr\"odinger equation \eqref{main_eq} admits a unique bounded outgoing solution
of the form
\[
u = \int_{\mathbb R^2} H^{(1)}_0(\kappa|x-y|) f {\rm d}y + h
\]
 with the following resolvent estimate
\begin{align}\label{res}
\|h\|_{L^\infty(\mathbb R^2)} = \mathcal{O} \Big(\frac{1}{\kappa^{1/2}}\Big).
\end{align}
Moreover, one has
\begin{align}\label{estimate}
\|u\|_{H^2_{loc}(\mathbb R^2)} \leq C\|f\|_{L^2(B_R)}.
\end{align} 
\end{corollary}

\section{Inverse source problem}\label{ip}

In this section we study the inverse source problem \textbf{IP}. 

Fix $\epsilon>0$ and $x_0\in\partial B_R$.
Consider the following equation with a perturbed source function 
\begin{equation}\label{eqn1}
-\Delta \tilde{u}_\epsilon - \kappa^2 \tilde{u}_\epsilon + \alpha \tilde{u}^2_\epsilon = f + \epsilon \delta(x - x_0).
\end{equation}
Decompose the solution $\tilde{u}_\epsilon$ as follows
\[
\tilde{u}_\epsilon = \epsilon H_0^{(1)}(\kappa|x - x_0|) + u_\epsilon
\]
where $H_0^{(1)}(\kappa|x - x_0|)$ is the fundamental solution satisfying
\[
-\Delta H_0^{(1)}(\kappa|x - x_0|) - \kappa^2 H_0^{(1)}(\kappa|x - x_0|) = \delta(x - x_0).
\]  
Then from the equation \eqref{eqn1}
one has that $u_\epsilon$ satisfies
\begin{equation}\label{eqn2}
-\Delta u_\epsilon - \kappa^2 u_\epsilon + \alpha u_\epsilon^2 + 2 \epsilon \alpha H_0^{(1)} u_\epsilon + \epsilon^2 \alpha (H_0^{(1)})^2  = f.
\end{equation}
As $x_0\in\partial B_R$ and $\text{supp}\, \alpha\subset B_R$, one has that $\alpha H_0^{(1)} \in L^\infty (B_R)$. Thus, by \eqref{estimate} one has
\[
\|u_\epsilon\|_{H^2(B_R)} \leq C.
\]
Denote
%\[
%u_\epsilon= u + \epsilon v_\epsilon ,
%\]
\[
v_\epsilon =\frac{u_\epsilon - u}{\epsilon} ,
\]
and subtracting \eqref{eqn2} by \eqref{eqn} and dividing by $\epsilon$ gives
\begin{equation}
-\Delta v_\epsilon - \kappa^2 v_\epsilon + \alpha v_\epsilon (\tilde{u}_\epsilon + u - \epsilon H_0^{(1)}) = -2 \alpha H_0^{(1)} u_\epsilon - \epsilon \alpha (H_0^{(1)})^2.
\end{equation}
Since both $\tilde{u}_\epsilon$ and $u$ are bounded, one has from \eqref{estimate} that
\begin{align*}
\|v_\epsilon\|_{H^2(B_R)} \leq C.
\end{align*}
Thus, letting $\epsilon\to 0$ one has that $v_\epsilon$ converges to some $v\in H^2(B_R)$ weakly which satisfies the following linear equation
\begin{equation}\label{eqn3}
-\Delta v - \kappa^2 v + 2 \alpha u (v + H_0^{(1)}) = 0.
\end{equation}
Moreover, one has $v_\epsilon \to v$ in the norm of $H^1(B_R)$.
Let $w = v + H_0^{(1)}$. Then the equation \eqref{eqn3} becomes
\begin{equation}\label{eqn4}
-\Delta w - \kappa^2 w + 2 \alpha u w = \delta (x - x_0),
\end{equation}
where $H_0^{(1)}$ is the point source incident field, $v$ is the outgoing scattered field and $w$ is the total field.
One also has all possible point-source boundary measurements of the scattered field $v$ as follows
\[
\{v(x, x_0): \text{for all} \, x, x_0\in\partial B_R\}.
\]

Now the inverse source problem is reduced to the 
inverse scattering problem of determining the unknown potential $2 \alpha u$ in the linear Schr\"odinger  equation
\[
-\Delta w - \kappa^2 w + 2 \alpha u w = \delta (x - x_0),
\]
from all possible point-source boundary measurements of the scattered field $v$
\begin{equation}\label{data1}
\{v(x, x_0): \text{for all} \, x, x_0\in\partial B_R\}.
\end{equation}
Indeed, once $2 \alpha u$ is determined, as $\alpha$ is given we can recover $u$, which
gives the recovery of $f$ by the equation $f = -\Delta u - \kappa^2 u + \alpha u^2$.

In what follows, we use results from the inverse potential scattering problem for linear Schr\"odinger equations to recover $2 \alpha u$.
In fact, if the following assumption is satisfied

 (A): $\kappa^2$ is not the Dirichlet eigenvalue of the elliptic operator $-\Delta + 2 \alpha u$ in $B_R$,

\noindent then using \cite[Corollary 1.4]{IN}, we can uniquely determine the unknown potential function $u$ by all possible point source boundary measurements \eqref{data1}, which gives the uniqueness of the inverse source problem.

Introduce the functional space
\[
\mathcal {C}_Q = \{f \in L^2(\mathbb R^2): \|f\|_{L^2(\mathbb R^2)}\leq Q, \text{supp}f \subset B_R, f: B_R\rightarrow \mathbb R\}.
\]
where $Q$ is some positive constant. We now discuss situations where the assumption (A) holds.  
For example, if the support of $\alpha$ is sufficiently small, we may choose the radius
$R$ of the disk $B_R$ sufficiently small such the first Dirichlet eigenvalue $\lambda_1$ of the elliptic operator $-\Delta$ in $B_R$ is large with $\kappa^2<\lambda_1$. Moreover, we may also have that the first Dirichlet eigenvalue $\tilde{\lambda}_1$ of $-\Delta + 2 \alpha u$ satisfies $\kappa^2<\tilde{\lambda}_1$ if the $L^\infty$ norm of the perturbation $2\alpha u$ is small. 
To this end, from \eqref{estimate} we can choose $Q$ to be sufficiently small which yields small $L^\infty$ norm of $u$.
We also point out that there may be other assumptions than the assumption (A) which lead to the uniqueness of $2 \alpha u$. 
For example, we may assume that $2 \alpha u$ is close to constant or has a small $H^2$ norm.
We refer the reader to \cite{Nov, Sun} for an account of this topic.

In summary, we arrive at the following main result of this paper. 

\begin{theorem}\label{main_1}
Assume that $|\kappa| > C_0$ as specified in Theorem \ref{bs} and $f\in \mathcal {C}_Q$. The boundary measurements \eqref{data} uniquely determine $f$ under assumption (A).
\end{theorem}

\section{Conclusion}

In this paper, we consider an inverse source scattering problem for a Schr\"odinger equation with a quadratic nonlinearity at a fixed energy in two dimensions. 
By adding reference point sources located on the measurement boundary we prove the uniqueness of the inverse source problem.  
This method does not apply to inverse source problems of linear equations as it utilizes the nonlinearity as a tool.
The proof utilizes the method of linearization to reduce the inverse source scattering problem to the inverse potential scattering problem of the linear Schr\"odinger equation. 
The uniqueness result then follows from existing results for inverse potential scattering problems of linear Schr\"odinger equation.
A possible continuation of this work is to consider cubic or quartic nonlinearity. In these cases, the method of linearization may not directly lead to 
inverse potential problems for linear equations and new method shall be developed. Another interesting direction is to develop numerical methods to reconstruct
the source function. As linearization used in the proof of the theoretical uniqueness result is difficult to implement in numerics , we need different methods other than the linearization.

%\appendix
%
%\section{The free resolvent in $\mathbb R^2$}\label{2}


\begin{thebibliography}{99}





\bibitem{Monk}
R. Albanese and P. Monk, The inverse source problem for Maxwell's equations, Inverse Problems, 22 (2006),
1023--1035.



\bibitem{BLLT}
G. Bao, P. Li, J. Lin, and F. Triki, Inverse scattering problems with multi-frequencies, Inverse Problems, 31
(2015), 093001.


\bibitem{blz}
G. Bao, P. Li, and Y. Zhao, Stability for the inverse source problems in elastic
and electromagnetic waves, J. Math. Pures Appl., 134 (2020), 122--178.

\bibitem{Bao}
G. Bao, J. Lin, and F. Triki, A multi-frequency inverse source problem, J. Differential Equations, 249 (2010), 3443--3465.


\bibitem{BC}
N. Bleistein and J. K. Cohen, Nonuniqueness in the inverse source problem in acoustics and electromagnetics, J.
Math. Phys., 18 (1977), 194--201.


\bibitem{CIL}
J. Cheng, V. Isakov, and S. Lu, Increasing stability in the inverse source
problem with many frequencies, J. Differential Equations, 260 (2016),
4786--4804.

\bibitem{Hu}
M. Choulli, G. Hu and M. Yamamoto, Stability estimate for a semilinear elliptic
inverse problem, Nonlinear Differ. Equ. Appl., 37 (2021).



\bibitem{FY06}
D.  Finco and K. Yajima, The $L^p$ boundedness of wave operators for Schr\"{o}dinger operators with threshold singularities, II. Even dimensional case, J. Math. Sci. Univ. Tokyo, 13 (2006), 277--346. 

\bibitem{IL}
V. Isakov and S. Lu, Increasing stability in the inverse source problem with attenuation and many frequencies, SIAM J. Appl. Math., 78 (2018), 1--18.


\bibitem{IN}
V. Isakov and A. Nachman, Global uniqueness for a two-dimensional elliptic inverse problem, Trans. AMS., 347 (1995), 3375--3391.


\bibitem{LZZ}
P. Li, J. Zhai and Y. Zhao, Stability for the acoustic inverse source problem in inhomogeneous media, SIAM J. Appl. Math., 80 (2020), 2547--2559.

\bibitem{LSX}
S. Lu, M. Salo and B. Xu, Increasing stability in the linearized inverse Schr\"odinger potential problem with power type nonlinearities, Inverse Problems, 38 (2022), 065009.


\bibitem{Myers}
J. Myers, uniqueness of source for a class of semilinear elliptic equations, Inverse Problems, 25 (2009), 065008.

\bibitem{Nov}
R. G. Novikov, The inverse scattering problem on a fixed energy level for the two dimensional
Schrodinger operator, J. Funct. Anal., 103 (1992), 409--463.

\bibitem{Servo}
V. Servo and M. Harju, A uniqueness theorem and reconstruction of singularities for a two-dimensional nonlinear Schr\"odinger equation, 
Nonlinearity, 21 (2008), 1323--1337.

\bibitem{Sun}
Z. Sun, On an inverse boundary value problem in two dimensions, Comm. Partial Differential
Equations, 14 (1989), 1101--1113.







\end{thebibliography}
\end{document}